\documentclass[12pt,a4paper]{amsart}

\usepackage{latexsym}

\usepackage{amssymb}
\usepackage{amsthm}
\usepackage{amsfonts}
\usepackage{amsmath}
\usepackage{amstext}
\usepackage{amscd}
\usepackage[dvips]{graphics}
\usepackage{epsfig}

\numberwithin{equation}{section}

\theoremstyle{plain}%default
\newtheorem{thm}{Theorem}[section]

\newtheorem{lem}[thm]{Lemma}

\newtheorem{theorem*}{Theorem}[]

\theoremstyle{definition}
\newtheorem{defn}[thm]{Definition}

\theoremstyle{remark}
\newtheorem{rem}[thm]{Remark}

\newcommand{\C}{\mathbb{C}}
\newcommand{\N}{\mathbb{N}}
\newcommand{\R}{\mathbb{R}}

\newcommand{\inv}{^{-1}}

\DeclareMathOperator{\ord}{ord\,}

\DeclareMathOperator{\grad}{grad\,}

\title[Topological equivalence] {A criterion for topological equivalence of two variable 
complex analytic function germs}

\author{Adam Parusi\'nski}

\address {Laboratoire Angevin de Recherche en Math\'ematiques, UMR 6093 
du CNRS,  Universit\'e d'Angers,
   2, bd Lavoisier, 49045 Angers cedex, France}

\email{adam.parusinski@univ-angers.fr}


 
\begin{document}

\begin{abstract} 
We show that two analytic function germs $(\C^2,0) \to (\C,0)$  are topologically right equivalent if and only if  
there is a one-to-one correspondence between the irreducible components
 of their zero sets that  preserves the multiplicites of these components,  
their Puiseux pairs, and the intersection numbers of any pairs of 
distinct components.  
\end{abstract}

\maketitle

\maketitle

%%%%%%%%%%%%%%%%%%%%%%%%%%%%%%%%%%%%%%%%%%%%%%%%%%%%%%%%%

%\section{Introduction}
%\label{intro}
%\medskip

By Zariski \cite{zariski} and Burau \cite{burau}, the topological type of an embedded 
 plane curve singularity  $(X , 0)\subset  (\C^2,0)$ is determined by the Puiseux pairs of each 
irreducible component (branch) of this curve  and the intersection numbers of any pairs of 
distinct branches.  In this note we show the following

\begin{thm}\label{theorem}
Let $f$, $g : (\C^2,0) \to (\C,0)$ be (not necessarily reduced) analytic function germs. 
Then $f$ and $g$ are topologically right equivalent if and only if  
there is a one-to-one correspondence between the irreducible components
 of their zero sets that  preserves the multiplicites of these components,  
their Puiseux pairs, and the intersection numbers of any pairs of 
distinct components.  
\end{thm}

\begin{proof} [Sketch of the proof]
The "only if" follows from the above cited result of Zariski and Burau.  

To show "if" we proceed as follows.  
%Denote the zero sets of $f$ and $g$ by $X_f$ and $X_g$ respectively.  
We may connect the zero sets $(f\inv (0),0)\subset (\C^2,0) $ and $(g \inv (0),0)\subset (\C^2,0) $
 by an  equisingular (topologically trivial) deformation of plane curve germs 
$$
(F\inv (0),0)\times P\subset (\C^2,0)\times P,$$
 where $P$ is a parameter space and $F : (\C^2,0)\times P \to (\C,0)$ 
is analytic.  Then,  by \cite{zariski2}  section 8, 
the pair $(F\inv (0)\setminus 0\times P,0\times P)$ satisfies Whitney conditions.  
Consequently, by \cite {BMM} or \cite{wf}, the strata 
$(\C^2,0)\times P\setminus F \inv (0)$, $F \inv (0)\setminus 0\times P$,  and $0\times P$ stratify 
$F$ as a function with the strong Thom condition $w_f$.  
This shows, by Thom-Mather theorem, that 
$F$ is topologically trivial along $P$.  
 \end{proof}

\begin{rem}
The following result was proven in  \cite{king} and \cite{saeki} by different arguments.    
Let $f$, $g : (\C^n,0) \to (\C,0)$, $n\ge 2$, be isolated analytic singularities  such that the 
germs $(\C^n, f\inv (0),0)$ and $(\C^n, g\inv (0),0)$ are homeomorphic.  Then either $f$ or $\overline f$ is 
 topologically right equivalent to $g$.   
  \end{rem}

\begin{rem}
In the isolated singularity case our proof can be expressed in a slightly different way.  By Zariski, c.f. \cite {zariski1}, an equisingular family of plane curves is $\mu$-constant and multipliciy constant, i.e. $\mu^*$ constant in the sense of Teissier \cite {teissier}.    It is known, c.f. \cite {wf}, \cite {BMM}, and the bibliography quoted therein,  that a $\mu^*$ constant 
 family of isolated singularities satisfies $w_f$ and hence is 
topologically trivial.   
\end{rem}

\begin{rem}
In \cite {kuolu} Kuo and Lu introduced a tree model $T(f)$ of 
an isolated singularity $f :(\C^2,0) \to (\C,0)$.  This model allows one to visualise 
 the Puiseux pairs of irreducible components of $f\inv (0)$ and  
the contact orders between them.  Kuo and Lu's model  can be easily addapted to the non-isolated case 
by adding the multiplicities of the components.   Then 
 theorem \ref{theorem} says that  $f$ and $g$ are topologically right  equivalent if and only if their 
tree models coincide.  
\end{rem}

We  give now details.  First we  connect $f$ by an equi\-singular deformation to a normal family that depends only on the embedded topological type of $(f\inv (0),0)\subset (\C^2,0) $ 
and the multiplicities of its branches.  
A similar construction gives an equisingular deformation connecting $f$ and $g$.
For these deformations we give elementary proofs of Whitney and strong Thom conditions and explicit formulae for 
vector fields 
trivializing them.

%%%%%%%%%%%%%%%%%%%%%%%%%%%%%%%%%%%%%%%%%%%
\medskip
\subsection{Deformation of $f$ to a normal family.}
Fix  $f:(\C^2,0)\to (\C,0)$.   Choose a system of coordinates so that $y=0$ is transverse to the tangent cone to 
$f=0$ at the origin.  Let $f_1, \ldots, f_N$ be the irreducible factors of $f$.  We write  
\begin{equation}\label{weierstrass}
f(x,y) = \prod_{k=1}^N (f_k (x,y))^{d_k}  =  u(x,y) \prod_{k=1 }^N \prod_{j=1}^{m_k}  
(x-\lambda_{k,j}(y))^{d_k} ,
\end{equation}
where $u(x,y)$ is a unit, $u(0,0)\ne 0$,  and    $ x=  \lambda_{k,i} (y)$  are Newton-Puiseux roots of 
$f_k$.  Then 
$$
 \lambda_{k,i} (y) = \sum a_\alpha (k,i) y ^\alpha , 
 $$
 are fractional power series: $\alpha \in \frac 1 {m_k} \N$.   
The  coefficients $a_\alpha (k,i)$ are well-defined if we restrict $\lambda$ 
 to a real half-line  through the origin.  In what follows we choose $y\in \R, y\ge 0$.  
 This allows us to define \emph{the contact order} between two such roots as 
 $$
O   (\lambda_{k_1,i}, \lambda_{k_2,j})
:= \ord_0 (\lambda_{k_1,i}-\lambda_{k_2,j})(y), \qquad y\in \R, y\ge 0.
$$  
All roots of $f_k$ can be obtained from  $ \lambda_{k,1} $ by 
$$
\lambda_{k,j} (y) %= \lambda_{k,1} (\theta^j y) 
= \sum a_\alpha (k,1) \theta^{j \alpha m_k} y ^\alpha ,
$$
where $\theta = e^{ {2\pi i/ m_k}}$.

 Denote by $\Lambda_{all,k} = \{\alpha_j\}$ the set of all contact orders  between 
$\lambda_{k,1}(y) $   and the other Newton-Puiseux roots of $f$.  
The Puiseux exponents of $\lambda_{k,j}(y) $  form a subset 
$ \Lambda_{P,k} \subset \Lambda_{all,k} $.   
Clearly $a_\alpha(k,j)\ne 0$ if $\alpha\in \Lambda_{P,k} $.  
The other exponents of  $ \Lambda_{all,k} $ can be divided into two groups.  If  
 the denominator of $\alpha$ does not divide the greatest common multiple of 
the denominators of $\alpha'\in \Lambda_{P,k}$, $\alpha'<\alpha$,  then $a_{\alpha}(k,j) = 0$, 
since otherwise $\alpha \in \Lambda_{P,k}$.   For the remaining exponents  there is no condition  
on the coefficient $a_{\alpha}(k,j)$, so we denote their set by 
$\Lambda_{free,k}$.  Finally we set 
$$\Lambda_k:= \Lambda_{P,k} \cup  \Lambda_{free,k}.$$  

For fixed $k$ we order  the roots $\lambda_{k,j}, j=1,\ldots,m_k$, by the 
lexicographic order on the sequences $( \arg (a_\alpha (k,j)), \alpha \in \Lambda_{P,k})$.  
Here $\arg \in [0,2\pi)$.  We obtain exactly the same ordering if we use the lexicographic 
order on $( \arg (a_\alpha (k,j)), \alpha \in \Lambda_{k}, a_\alpha (k,j) \ne 0)$ (for two conjugate 
roots $O (\lambda_{k,i} , \lambda_{k,j})$ is a Puiseux exponent). 
We reorder the roots so that  $\lambda_{k,1}(y) $ is the smallest among all $\lambda_{k,i}(y) $ 
with respect to this order.

\begin{lem}
Let $k_1\ne k_2$.  Then 
%For every $i=1,\ldots, m_{k_1}$ and $j= 1,\ldots,m_{k_2}$ 
$$
\max_{i,j} \, O (\lambda_{k_1,i}, \lambda_{k_2,j})   = O (\lambda_{k_1,1}, \lambda_{k_2,1}) .
$$
\end{lem}

\begin{proof}
Suppose, contrary to our claim, that  $O (\lambda_{k_1,1}, \lambda_{k_2,1}) < O (\gamma_{k_1,i}, \gamma_{k_2,j}) $.  Then there exist $j'$ and  $i'$ such that  
$$
O (\lambda_{k_1,i}, \lambda_{k_2,j}) = O (\lambda_{k_1,1}, \lambda_{k_2,j'}) = 
O (\lambda_{k_1,i'}, \lambda_{k_2,1}) . 
$$
Then for $\alpha = O (\lambda_{k_1,1}, \lambda_{k_2,1})$,  
$ a_\alpha (k_1,1)=  a_\alpha (k_2,j')\ne 0$ and $ a_\alpha (k_1,i')=  a_\alpha (k_2,1)\ne 0$, 
and hence 
$$
\arg( a_\alpha (k_1,1)) =  \arg (a_\alpha (k_2,j')) > \arg (a_\alpha (k_2,1))
$$
and 
$$
\arg( a_\alpha (k_2,1)) =  \arg (a_\alpha (k_1,i')) > \arg (a_\alpha (k_1,1))
$$
that is impossible.  
\end{proof}

\medskip
\begin{defn}
\emph{The deformation space $D(f) \subset \prod_k \C^{ |\Lambda_k|}$} is defined as follows.  
Write an element of 
$ \C^{ |\Lambda_k|}$ as $a(k)= (a_\alpha (k); \alpha \in \Lambda_k) 
\in  \C^{ |\Lambda_k|}$.  Then, $\underline a =  (a(k); k=1,\ldots, N) \in D(f)$ if : 
\begin{enumerate}
\item
 if $\alpha \in \Lambda_{P,k}$ the $a_{\alpha}(k) \ne 0$,
\item
if $\alpha < O (\lambda_{k_1,1}, \lambda_{k_2,1}) $ and 
$\alpha \in \Lambda_{k_1} \cap \Lambda_{k_2}$ 
then $a_{\alpha}(k_1) =a_{\alpha}(k_2)$.  
\item
if $\alpha = O (\lambda_{k_1,1}, \lambda_{k_2,1}) $ and 
$\alpha \in \Lambda_{k_1} \cap \Lambda_{k_2}$ 
then $a_{\alpha}(k_1) \ne a_{\alpha}(k_2)$.  
\end{enumerate}
\end{defn}

\medskip
Write each Newton-Puiseux root of $f$ as 
 \begin{equation}\label{expansions}
 \lambda_{k,j} (y) =  \sum_{\alpha \ge 1}  a_\alpha (k,j) y ^\alpha 
 = \sum_{\alpha\in \Lambda_k}  a_\alpha (k,j) y ^\alpha  + R_{k,j}(y) .  
 \end{equation}
For  $(s,\underline a) \in \C \times   D(f)$ write 
$$
 \lambda_{k,j} (s, \underline a, y) 
=  \sum_{\alpha\in \Lambda_k}  a_\alpha (k)   \theta_k^{j\alpha m_k} y^\alpha  + s R_{k,j}(y),
$$
where $\theta_k = e^{ {2\pi i/ m_k}}$.  
Consider the following deformation of $f$ 
\begin{equation}\label{deformation}
F (\tau,s,u_0,\underline a, x,y) = 
(\tau ( u(x,y)- u(0,0)) + u_0) \prod_{k=1 }^N \prod_{i=1}^{m_k}  
(x-   \lambda_{k,i} (s, \underline a, y)  )^{d_k} ,
\end{equation}
where $(\tau,s,u_0,\underline a) \in P: = \C\times \C \times  \C^*\times D(f)$.  ($P$ stands for the 
parameter space).  Then $F$ is analytic in all variables and 
 $f(x,y) = F (1,1,u(0,0),\underline a(f), x,y) $, where  $\underline a(f)$ is 
  given by the coefficients $a_\alpha (k,1)$ of the Newton-Puiseux roots of $f$.  
  
  We call $F : D(f) \times (\C^2,0)\to (\C,0)$ \emph{the normal family of 
  germs associated to  $f$}.   It depends only on the topological type of 
$(f\inv (0),0) \subset  (\C^2,0)$ and the multiplicities $d_k$ of irreducible components of $f$.

We may as well embedd $f$ and $g$ in one equisingular family by taking 
\begin{equation}\label{fgdeformation}
F (\tau,s,u_0,\underline a, x,y) = 
(\tau_1  u_1(x,y) + \tau_2 u_2(x,y) + u_0) \prod_{k=1 }^N \prod_{i=1}^{m_k}  
(x-   \lambda_{k,i} (s, \underline a, y)  )^{d_k} ,
\end{equation}
where $u_1 = u_f (x,y) - u_f(0,0)$ and $u_2 = u_g (x,y) - u_g (0,0)$ and $u_f$, resp. $u_g$, denote
 the unit of \eqref{weierstrass} for $f$ and $g$ respectively, $s=(s_f,s_g)$ and 
$$
 \lambda_{k,j} (s, \underline a, y) 
=  \sum_{\alpha\in \Lambda_k}  a_\alpha (k)   \theta_k^{j\alpha m_k} y^\alpha  + s_f R_{f,k,j}(y) 
+ s_g R_{g,k,j}(y).  
$$

\smallskip
%%%%%%%%%%%%%%%%%%%%%%%%%%%%%%%%%%%%%%%%%%%%%%

\subsection{Whitney Conditions and Thom Condition}
Let $X= F\inv (0) \subset U \subset P\times \C^2$, where $U$ is a small open neighbourhood of 
$P\times \{0,0\}$ in $P\times \C^2$.   
%We shall identify $P$ and $P\times \{0,0 \}$.  
We show by elementary computations that  $(U\setminus X, X\setminus P, P)$ as a stratification, 
satisfies Whitney condition and the strong Thom condtion $w_f$.

Let  $G$ be the reduced version of $F$:
\begin{equation*}%\label{reduced}
 G (\tau,s,\underline a, x,y) = u(\tau,x,y) \prod_{k=1 }^N \prod_{i=1}^{m_k}  
(x-  \lambda_{k,i} (s, \underline a, y)  ) ,
\end{equation*}
The pair 
$(X\setminus P, P)$ satisfies   Verdier condition $w$, \cite {verdier}, 
equivalent to Whitney conditions in the complex case, if for any $(x_0,y_0) \in P$ there is a constant $C$ such 
that on  $X$ near $(x_0,y_0)$ the following inequality holds  
  \begin{equation}\label{w}
| \frac {\partial G}{\partial v} |  \le C \|(x,y)\| 
\| ( \frac {\partial G}{\partial x}, \frac {\partial G}{\partial y})  \| 
\end{equation}
where $v$  denotes any of the coordinates on the parameter space $P$.   
 A direct computation on $x= \lambda_{k_0,i_0} (s, \underline a, y) $  gives 
\begin{eqnarray*}%\label{reduced}
& & \frac {\partial G}{\partial v } = (- \frac {\partial \lambda_{k_0,i_0}} 
{\partial v})\prod_{(k,i) \ne (k_0,i_0)}   
(x-  \lambda_{k,i} (s, \underline a, y)  ) , \\
& & \frac {\partial G}{\partial x} = \prod_{(k,i) \ne (k_0,i_0)}   
(x-  \lambda_{k,i} (s, \underline a, y)  ) .
% \\ \frac {\partial G}{\partial y } = (- \frac \partial \lambda_{k_0,i_0}} {\partial y}) 
% \prod_{(k,i) \ne (k_0,i_0)}   (x-  \lambda_{k,i} (s, \underline a, y)  ) .
\end{eqnarray*}
and \eqref{w} follows from 
\begin{equation}\label{useful}
|\frac {\partial \lambda_{k_0,i_0}} {\partial v}|\le C |y|, 
\end{equation}
that is a consequence of the fact that all exponents in 
$\bigcup_k \Lambda_k$ are $\ge 1$.

Similarly, $(U\setminus X, X\setminus P, P)$ as a stratification of $F$ satisfies the strong Thom condition $w_f$ 
 if for any $(x_0,y_0) \in P$ there is a constant $C$ such that in a neighborhood of 
  $(x_0,y_0)$ the following inequality holds  
  \begin{equation}\label{wf}
| \frac {\partial F}{\partial v} |  \le C \|(x,y)\| 
\| ( \frac {\partial F}{\partial x}, \frac {\partial F}{\partial y})  \| 
\end{equation}
in the complement of the zero set of $F$.  

Replace $y$ by $t^n$ so that the roots 
$\lambda_{k,i} (s, \underline a, t^n) $ become analytic.  We denote them using a single index as 
$\tilde \lambda_{m} (s, \underline a, t) $.  Given a power series 
$\xi (v,t) = \sum_{\alpha =1}^\infty a_\alpha  (v) t^\alpha$, where $v\in P$,  we consider 
\emph {a horn-neighborhood of $\xi$} 
$$
H_d (\xi, w) = \{ (x,v,t) ; |x- \xi (v,t)|\le w |t|^d \} ,
$$
where $d\in \N\cup \{0\}$ and $w>0$, compare  \cite{kuolu} $\S$ 6.  
In what follows $\xi (v,t) $ will be one of the roots $\tilde \lambda_m$.  
In $H_d (\xi, w)$ we use the coordinates $\tilde x,t,v$, where 
$$
x= \tilde xt^d + \xi(v,t).  
$$
Let $I(\xi,d)= \{m; \, O(\tilde \lambda_{m} , \xi) \ge d\}$.  Then 
\begin{equation}\label{inhorn}
F(v,\tilde x, t) = u(v,\tilde x, t) \, t ^M \prod_{m\in I(\xi,d)} (\tilde x -  \beta_{m} ( v, t))^{d_m}, 
\end{equation}
where $M= \sum_{m\notin I(\xi,d)} d_m O(\tilde \lambda_{m},\xi)  + d \sum_{m\in I(\xi,d)} d_m $,  $u$ is a unit, 
and $ \beta_{m}  = t^{-d}(\tilde  \lambda_{m} -\xi)$.

Denote $e_{m,j} = O (\tilde\lambda_m,\tilde\lambda_j)$ for $j\ne m$.  Let  $C>0$ stands  for a large 
constant and $\varepsilon > 0$ for a small constant. 

\medskip
\begin{lem}\label{firstcase}
For  $d\in \N\cup \{0\}$ arbitrary, 
 \eqref{wf} holds on 
$\hat H_d (\tilde \lambda_m, \varepsilon ,C) = H_d (\tilde \lambda_m, \varepsilon) \setminus \bigcup_l 
H_{e_{m,l}}(\tilde \lambda_m, C)$, 
where the union is taken over all $l\ne m$ such that $O(\tilde\lambda_m, \tilde \lambda_l) > d $.  
\end{lem}

\begin{proof}
By \eqref{inhorn} 
\begin{equation}\label{dv}
\frac {\partial F} {\partial v} =  \Bigl ( \sum _j  d_k \frac {-\partial \beta_j /\partial v} 
{\tilde x -\beta_j (t,v)}  + \frac {\partial u/\partial v } u  \Bigr ) F 
\end{equation}
$$
\frac {\partial F} {\partial x} = t^{-d} \, \frac {\partial F} {\partial \tilde x} = 
\Bigl ( \sum _j  d_k \frac {1} 
{\tilde x -\beta_j (t,v)} +   \frac {\partial u/\partial \tilde x } u \Bigr ) \, t^{-d} \,  F   . 
$$
If $( \tilde x,t,v) \in  \hat H_{d}( \tilde \lambda_m, \varepsilon,C) $ and $O(\tilde\lambda_i, \tilde \lambda_l) > d $
 then 
$|\beta _i - \beta _ l | \le \delta |\tilde x -\beta_i |$, for $\delta$ small if $C$ is large,  and hence 
$$
|\frac {\partial F} {\partial \tilde x}| \ge \frac \delta {| \tilde x -\beta_i |} \, | F | .
$$
Then \eqref{wf} follows from \eqref{useful} since
$|\frac {\partial \beta_j} {\partial v}|\le C |t^{-d}y|  $.  
\end{proof}

\begin{lem}\label{secondcase}
Fix $d=e_{m,j}$.  Then 
\eqref{wf} holds on 
$\tilde H_{d} ( \tilde \lambda_m, C, \varepsilon )  = H_{d} ( \tilde \lambda_m, C) \setminus \bigcup_l H_{d}( \tilde \lambda_l, \varepsilon)$, 
where the union is taken over all $l$ such that $O( \tilde\lambda_m, \tilde \lambda_l) \ge d$ 
(including $l=m$).    
\end{lem}

\begin{proof}
Note that on  $\tilde H_{d} (\tilde \lambda_m, C, \varepsilon )$, $ |\tilde x -\beta_j (t,v)| \ge \varepsilon$ and 
$ \partial \beta_j /\partial t$ is bounded.  Therefore 
\begin{equation*}%\label{dt}
|t \frac {\partial F} {\partial t}| = |   \Bigl (M+ \sum _j  d_j \frac {-t \, \partial \beta_j /\partial t} 
{\tilde x -\beta_j (t,v)}  +  \frac {t\partial u/\partial t } u \Bigr ) \,F| \ge (M-\delta) \cdot |F|
\end{equation*}
where $\delta\to 0$ as $t\to 0$.  Since $M>0$,  we have after \eqref{dv}
\begin{equation}\label{wfbound2}
|\frac {\partial F} {\partial v}| \le C |F|   \le |t \frac {\partial F} {\partial t}| \le  |  d (y \frac {\partial F} {\partial x} 
+ y \frac {\partial F} {\partial y} )|
\end{equation} 
\end{proof}

To complete the proof of \eqref{wf} we note that a finite family of horns $\hat H_d (\tilde \lambda_m, \varepsilon ,C)$, 
$\tilde H_{d} ( \tilde \lambda_m, C)$, with $d = e_{m,j}$ or $d=0$, covers a neighborhood of the origin in $\C^2$ (times 
a neighborhood in the parameter space). 

%%%%%%%%%%%%%%%%%%%%%%%%%%%%%%%%%%%%%%%%%%%%%%%%%%%%%

\smallskip
%%%%%%%%%%%%%%%%%%%%%%%%%%%%%%%%%%%%%%%%%%%%%%

\subsection{Trivialization}

By  Thom-Mather theory  the family $F$ is locally topologically trivial along $P$.  
Moreover, it can be show easily that it can be trivialized by Kuo's vector field, cf. \cite{kuo}. 
Suppose that the parameter $P$ 
space is one-dimensional with the parameter $v$.  Define the vector field 
$V$ on each stratum separately by formulae :
\begin{eqnarray*}
&  V: = \frac \partial {\partial v}    & \qquad \text { on $P$ } \\
&  V: = \frac \partial {\partial v} - \frac{  { \partial G} /{\partial v}}  {\|\grad_{x,y} G\|^2} \Bigl(
\overline { \frac {\partial G} {\partial x}} \frac {\partial } {\partial x} +
 \overline { \frac {\partial G} {\partial y}} \frac {\partial } {\partial y}
\Bigr ) 
 &  \qquad \text { on $X\setminus P$ } \\
 & V: = \frac \partial {\partial v} - \frac{  { \partial F} /{\partial v}}   {\|\grad_{x,y} F\|^2} \Bigl(
\overline { \frac {\partial F} {\partial x}} \frac {\partial } {\partial x} +
 \overline { \frac {\partial F} {\partial y}} \frac {\partial } {\partial y}
\Bigr ) 
& \qquad \text { on $U\setminus X$ } . 
\end{eqnarray*}
The flow of $V$ is continuous and preserves the levels of $F$ by the arguments of \cite{kuo}. \\
 
%%%%%%%%%%%%%%%%%%%%%%%%%%%%%%%%%%%%%%%%%%%%%%%%%%%%%%%%%

%%%%%%%%%

\end{document}